\documentclass[a4paper,11pt]{amsart}
\usepackage[plainpages=false]{hyperref}
\usepackage{amsfonts,latexsym,rawfonts,amsmath,amssymb,amsthm,mathrsfs}
\usepackage{amsmath,amssymb,amsfonts,latexsym,lscape,rawfonts}
\usepackage[usenames]{color}

\title[Convergence of the Calabi flow]
{Convergence of the calabi flow on toric varieties and related K\"ahler manifolds}

\author{Hongnian Huang}
\thanks{The research of the author is financially supported by FMJH (Fondation Math\'ematique Jacques Hadamard).}

\date{May 15, 2012}

\newtheorem{thm}{Theorem}[section]
\newtheorem{cor}[thm]{Corollary}
\newtheorem{lem}[thm]{Lemma}
\newtheorem{prop}[thm]{Proposition}

\newtheorem{conj}[thm]{Conjecture}
\newtheorem{rmk}[thm]{Remark}

\theoremstyle{definition}
\newtheorem{defn}{Definition}[section]

\numberwithin{equation}{section}

\makeatother
\begin{document}
\maketitle

\begin{abstract}
Let $X$ be a toric variety and $u$ be a normalized symplectic potential of the corresponding polytope $P$. Suppose that the Riemannian curvature is bounded by $1$ and 
$
\int_{\partial P} u ~ d \sigma < C_1,
$
then there exists a constant $C_2$ depending only on $C_1$ and $P$ such that $\max_P u < C_2$. As an application, we show that if $(X,P)$ is analytic uniform $K$-stable, then the modified Calabi flow converges to an extremal metric exponentially fast by assuming that the Riemannian curvature is uniformly bounded along the Calabi flow. Also we provide a proof of a conjecture of Donaldson. Finally, assuming that the curvature is bounded along the Calabi flow, our method would provide a proof of a conjecture due to Apostolov, Calderbank, Gauduchon and T{\o}nnesen-Friedman.
\end{abstract}

\section{Introduction}
Let $X$ be a K\"ahler manifold with K\"ahler class $[\omega]$. The candidates of the canonical metrics in $[\omega]$ are extremal metrics in the sense of Calabi \cite{Ca2}. Calabi \cite{Ca1} also proposes an analytic method to search for his extremal metrics. This analytic method is usually called the Calabi flow whose equation is
$$
\frac{\partial \varphi}{\partial t} = R_\varphi - \underline{R},
$$
where $\varphi$ is a K\"ahler potential, $R_\varphi$ is the corresponding scalar curvature and $\underline{R}$ is the average of the scalar curvature.

It is conjectured by Chen \cite{ChenHe3} that the Calabi flow exists for all time. When $X$ is a toric variety, it is shown in \cite{H1} that the obstructions of the long time existence are the regularity theorem and the non-collapsing property of the Calabi flow. Later, Streets establishes the regularity theorem\footnote{Chen and He established the weaker regularity theorem in \cite{ChenHe2}} of the Calabi flow in \cite{St}. Recently, Chen's conjecture has been confirmed in \cite{FH} when $X = \mathbb{C}^2 / \mathbb{Z}^2 + i \mathbb{Z}^2$ and the initial metric is invariant under the translation of the imaginary part of the complex variables. 

In this note, we are interested in the long time behavior of the Calabi flow on a toric variety. In fact, it is expected that the long time behavior of the Calabi flow should be related to the stability of the manifold. 

When $(X, L)$ is $K$-unstable, assuming the long time existence of the Calabi flow on a toric variety $X$, the geometrical phenomenas are clear by the work of Sz\'ekelyhidi \cite{S1}. More explicitly, assuming the long time existence of the Calabi flow on a toric variety, he proves that the infimum of the Calabi energy is equal to the supremum of the normalized Futaki invariants of all destabilizing test configurations, confirming a conjecture of Donaldson \cite{D7}.

The remaining case is that when $(X, L)$ is $K$-stable, if we assume the Calabi flow exists for all time, what will happen? We recall that Yau \cite{Y1}, Tian \cite{Ti1} and Donaldson \cite{D1} conjecture that the $K$-stability of $(X, L)$ is equivalent to the existence of cscK metrics. Thus we should try to understand the long time behavior of the Calabi flow by assuming that there exists a cscK metric in the K\"ahler class. In fact Donaldson \cite{D6} conjectures that if there exists a cscK metric in the K\"ahler class and if the Calabi flow exists for all time, then it should converge to a cscK metric. The work in \cite{FH} confirms Donaldson's conjecture when $X=\mathbb{C}^n / \mathbb{Z}^n + i \mathbb{Z}^n$. This work also suggests that if there exists a cscK metric in the K\"ahler class, then the curvature along the Calabi flow should be uniformly bounded.  Thus we modify Donaldson's conjecture as in \cite{H2}.

\begin{conj}
\label{c1}
Let $X$ be a K\"ahler manifold with K\"ahler class $[\omega]$. Suppose that there exists an extremal metric $\omega_0 \in [\omega]$. Let $\omega_1  \in [\omega]$ be a K\"ahler metric  invariant under the maximal compact subgroup of the identity component of the reduced automorphism group. We further assume that the curvature along the Calabi flow starting from $\omega_1$  is uniformly bounded. Then the modified Calabi flow converges to an extremal metric exponentially fast.
\end{conj}

Combining with the Yau-Tian-Donaldson conjecture and the recent developments in \cite{ACGT}, \cite{D5} and \cite{S2}, we have the following conjecture in \cite{H2} when $X$ is a toric variety.

\begin{conj}
\label{c2}
Let $X$ be a toric variety with an ample line bundle $L$. Suppose that $(X, L)$ is relative $K$-stable and the curvature along the Calabi flow starting from an toric invariant metric $\omega \in c_1(L)$ is uniformly bounded. Then the modified Calabi flow converges to an extremal metric exponentially fast.
\end{conj}

We can write down the $K$-stability condition explicitly when $X$ is a toric variety. Let $P$ be the Delzant polytope corresponding to $(X,L)$ through moment map. We also define the extremal function $\theta$ which is an affine function satisfying the following equation
\begin{eqnarray}
\label{extremal function}
\mathcal{L}(u) = 2 \int_{\partial P} u ~ d \sigma - \int_P u \theta ~ d \mu = 0
\end{eqnarray}
for all affine function $u$. The relative $K$-stability can be interpreted as in \cite{D1}:
\begin{defn}
$(X, L)$ is relative $K$-stable if 
$$
\mathcal{L} (f) \geq 0
$$
for all rational piecewise linear function $f$ and the equality holds if and only if $f$ is an affine function.
\end{defn}

There is another stability condition introduced by Sz\'ekelyhidi \cite{S3}. Let $\mathcal{C}_\infty$ be the set of continuous convex functions on $\bar{P}$ which are smooth in the interior. Let us fix a point $x_0 \in P$. For any $u \in \mathcal{C}_\infty$, we say $u$ is normalized if $u(x_0) = 0$ and $D u(x_0) = 0$. The uniform stability of Sz\'ekelyhidi states as follows:

\begin{defn} 
$(X, L)$ is uniform $K$-stable if there exists a constant $\lambda > 0$ such that for any normalized function $u \in \mathcal{C}_\infty$, we have
$$
\mathcal{L}(u) \geq \lambda ||u||_{L^{\frac{n}{n-1}}}.
$$
\end{defn}

For the purpose of analysis, we will use the analytic version of uniform stability from \cite{D1}. 

\begin{defn}
$(X, P)$ is analytic uniform $K$-stable if there exists a constant $\lambda > 0 $ such that for any normalized $u \in \mathcal{C}_\infty$, we have
$$
\mathcal{L}(u) \geq \lambda \int_{\partial P} u ~ d \sigma.
$$
\end{defn}

When the complex dimension of $X$ is one, Conjecture (\ref{c2}) is proved by Chen \cite{Ch1} and Chen-Zhu \cite{CZ}. When $X$ is a toric surface, Conjecture (\ref{c2}) is proved in \cite{H2} if $(X, P)$ is analytic relative $K$-stable. In this note, we extend the above mentioned results by removing the dimensional constraints. Our main theorem states as follows:

\begin{thm}
\label{main}
Let $(X,P)$ be analytic uniform $K$-stable. Suppose that the curvature is uniformly bounded along the Calabi flow starting from a toric invariant metric in the K\"ahler class of $(X,P)$. Then the modified Calabi flow converges to an extremal metric exponentially fast.
\end{thm}

\begin{rmk}
The existence of extremal metrics in a toric variety implies that $(X, P)$ is analytic uniform stable by the work of Chen-Li-Sheng \cite{CLS}. Thus we also prove Conjecture (\ref{c1}). 
\end{rmk}

\begin{rmk}
Tosatti has a related result in K\"ahler Ricci flow \cite{T1}.
\end{rmk}

Assuming the curvature of the Calabi flow is uniformly bounded, our method also leads us to provide a proof of a conjecture due to Apostolov, Calderbank, Gauduchon and T{\o}nnesen-Friedman \cite{ACGT2}. The conjecture states as follows:

\begin{conj}[\cite{ACGT2}]
\label{cos}
A projective bundle $(X, J) = P(E)$ over a compact curve $\Sigma$ of genus $\geq 2$ admits an extremal metric in some K\"ahler class if and only if $E$ decomposes as
$$
E = \bigoplus_{i=0}^l E_i,
$$
where $E_i$ is a stable subbundle, $1 \leq i \leq l$.
\end{conj}

When $l=1$, the conjecture is proved in \cite{ACGT}. In \cite{ACGT2}, the authors show that to prove Conjecture (\ref{cos}), we only need to show that if $\Omega$ is a compatible class on $X$, then the existence of an extremal K\"ahler metric in $\Omega$ implies the existence of a compatible extremal K\"ahler metric in $\Omega$. Our following theorem suggests a way to prove Conjecture (\ref{cos}).

\begin{thm}
\label{proof of cos}
Suppose there exists an extremal metric $\omega_0 \in \Omega$. Let $\omega \in \Omega$ be a compatible K\"ahler metric. If the curvature is uniformly bounded along the Calabi flow starting from $\omega$, then the modified Calabi flow converges to a compatible extremal metric $\omega_1 \in \Omega$ exponentially fast.
\end{thm}

\begin{rmk}
Theorem (\ref{proof of cos}) can be extended to other situations:

\begin{enumerate}
\item For any manifold $X$ which is a rigid toric bundle over a semisimple cscK base in \cite{ACGT2}. 

\item For any manifold $X$ which is multiplicity-free in the sense of Donaldson \cite{D8} and Raza \cite{Ra}.
\end{enumerate}

\end{rmk}

{\bf Acknowledgment: } The author would like to thank Vestislav Apostolov and  G\'abor Sz\'ekelyhidi for many stimulating discussions. He is also grateful to the consistent support of Professor Xiuxiong Chen, Pengfei Guan and Paul Gauduchon. He would like to thank Si Li, Jeff Streets and Valentino Tosatti for their interests in this work.

\section{Notations and Setup}

Let $X$ be a $n$-dimensional toric variety and $[\omega]$ be a K\"ahler class of $X$. Then we obtain the corresponding Delzant polytope $P$ through the moment map. The Delzant conditions show that

\begin{itemize}

\item For every facet $P_i$ of $P$, there exists an inward normal vector $\vec{n}_i$ corresponding to $P_i$.

\item For each vertex $v$ of $P$, there are exactly $n$ facets $P_{i_1}, \ldots, P_{i_n}$ meeting at $v$. Moreover, $\vec{n}_{i_1}, \ldots, \vec{n}_{i_n}$ form a basis of $\mathbb{Z}^n$.
\end{itemize} 

Suppose $P$ has $d$ facets. Let $l_i(x) = \langle x, \vec{n}_i \rangle + c_i, ~ i = 1, \ldots, d$ and we choose $c_i$ properly such that $l_i(x) = 0$ for all $x \in P_i$. The Guillemin boundary conditions show that for any toric invariant K\"ahler metric, its symplectic potential $u$ satisfies
\begin{itemize}
\item $u$ is a smooth, strictly convex function in $P$.

\item The restriction of $u$ to each face of $\partial P$ is smooth and strictly convex. 

\item
$$
u(x) = \frac{1}{2} \sum_{i=1}^d  l_i(x) \ln l_i(x) + f(x), ~ x \in P,
$$
where $f(x)$ is a smooth function on $\bar{P}$. 
\end{itemize}
The scalar curvature of $u$ on $P$ is the Abreu's equation
$$
R_u(x) = - \sum_{i j} u^{ij}_{~ij}(x), ~ x \in P.
$$

The Calabi flow equation on the symplectic side reads
$$
\frac{\partial u}{\partial t} = \underline{R} - R_u.
$$
Following \cite{HZ}, the modified Calabi flow equation on the symplectic side reads
$$
\frac{\partial u}{\partial t} = \theta - R_u,
$$
where $\theta$ is an affine function satisfying Equation (\ref{extremal function}).

The modified Mabuchi energy defined in \cite{D1} is 
\begin{eqnarray*}
\mathcal{M} (u) & = & -\int_P \log \det(u_{ij})~ d \mu + \mathcal{L}(u)\\
&=& -\int_P \log \det(u_{ij})~ d \mu + 2 \int_{\partial P} u ~ d \sigma - \int_P u \theta ~ d \mu,
\end{eqnarray*}
where $d \mu$ is the standard Lebesgue measure and $d \sigma$ is the standard Lebesgue measure divided by $|\vec{n}_i|$ for each facet $P_i$. It is known that the modified Calabi flow is the downward gradient flow of the modified Mabuchi energy (See e.g. \cite{H2}).

\section{$L^\infty$ Control}

In this section, we will prove

\begin{thm}
\label{dia}
Suppose $u$ is a normalized symplectic potential such that the Riemannian curvature is bounded by $1$ and there exists a constant $C_1$ such that
$$
\int_{\partial P} u ~ d \sigma < C_1.
$$
Then there exists a constant $C(C_1, P)$  depending only on $C_1$ and $P$ such that $\max_P u < C$.
\end{thm}

If $X=\mathbb{CP}^1$, the above theorem is trivially true. And when $X$ is a toric surface, our theorem is proved in \cite{H2}. Thus we only consider the case where the dimension of $X$ is greater than two. Without loss of generality, we can assume that we work in a standard model, i.e., $O = (0, \ldots, 0)$ is a vertex of $P$, $x_1, \ldots, x_n$ where $0 \leq x_i \leq 1$ for all $i$ are the edges of $P$ around $O$ and $P$ lies in the first quadrant. Guillemin's boundary conditions tell us that 
$$
u= \frac{1}{2} ( x_1 \ln x_1 + \cdots + x_n \ln x_n) + f(x_1, \ldots, x_n),
$$
where $f$ is a smooth function up to the boundary.

Our first observation is the following lemma:

\begin{lem} In each edge $x_i$, let $V_i(x_i) =  \frac{1}{2} x_i \ln x_i + f(0,\ldots,0,x_i,0,\ldots,0)$, then
$$
\left( \frac{1}{V_i''} \right)''(x_i) = \lim_{x \rightarrow (0,\ldots,0,x_i,0,\ldots,0) } u^{ii}_{~ii} (x).
$$
\end{lem}

\begin{proof}
Let 
$$
A =\det \left(
\begin{array}{ccc}
\frac{1}{2x_2} + f_{22} & \cdots & u_{2n} \\
& \vdots &\\
u_{n2} & \cdots & \frac{1}{2x_n} + f_{nn}
\end{array}
\right)
$$

Then
$$
u^{11} = \frac{A}{\det(u_{ij})}
$$

We obtain
\begin{eqnarray*}
u^{11}_{~11} & = & \left( \frac{A_{1}}{\det(u_{ij})} -   \frac{A \det(u_{ij})_1}{(\det(u_{ij}))^2} \right)_1 \\
&=& \frac{A_{11}}{\det(u_{ij})} -  2 \frac{A_1 \det(u_{ij})_1}{(\det(u_{ij}))^2} - \frac{A \det(u_{ij})_{11}}{(\det(u_{ij}))^2 } + 2 \frac{A (\det(u_{ij})_1)^2}{(\det(u_{ij}))^3}.
\end{eqnarray*}

Let $$V(x_1) = \frac{1}{2} x_1 \ln x_1 + f(x_1,0,\ldots,0), \quad v(x_1) = V''(x_1)$$ As $x \rightarrow (x_1, 0, \ldots, 0)$, we get
\begin{eqnarray*}
\lim_{x \rightarrow (x_1, 0, \ldots, 0)} u^{11}_{~11} (x) & = & \left( - \frac{v''}{v^2} + 2 \frac{v'^2}{v^3} \right) (x_1) \\
& = & \left(- \frac{v'}{v^2} \right)' (x_1) \\
& = & \left( \frac{1}{v} \right) ''(x_1).
\end{eqnarray*}

Hence we obtain the desired result.
\end{proof}

It is shown in \cite{D2} that the norm of Riemannian curvature is expressed as 
$$
|Rm|^2 = \sum u^{ij}_{~kl} u^{kl}_{~ij}.
$$

By direction calculations, we have

\begin{lem} $u^{ij}_{~kl}(x_1, \ldots, x_n)$ is finite for all $x_1, \ldots, x_n \in [0,1]$ and all $i,j,k,l$. Moreover for each $i,j,k,l$ and $x_1 \in (0,1]$.
$$
u^{ij}_{~kl} u^{kl}_{~ij} (x_1, 0, \ldots, 0) = 0
$$
unless $i=k, j=l$ or $i=l, j=k$.
\end{lem}

\begin{proof}
Expressing $u^{ij}_{~kl}(x_1, \ldots, x_n)$ out, we can see that it is finite. To calculate $u^{ij}_{~kl} u^{kl}_{~ij} (x_1, 0, \ldots, 0)$, we divide all the cases into 5 categories.

\begin{enumerate}
\item None of $i,j,k,l$ equals to 1. First, we assume that $i \neq j$. Expressing $u^{ij}_{~kl}$ out, the only case that $u^{ij}_{~kl}$ is not zero is $i = k, j = l$ or $i = l, j =k$. The remaining case is $u^{ij}_{~kl} = u^{ii}_{~jj} $ which is zero unless $i=j$.

\item There is exactly one of $i,j,k,l$ which equals to 1. We can assume that $l=1$. Then $u^{ij}_{~k} = 2$ if $i=j=k$ otherwise $u^{ij}_{~k} = 0$. In any case we obtain $u^{ij}_{~kl} u^{kl}_{~ij} = 0$.

\item There are exactly two of $i,j,k,l$ which equals to 1. There are essentially two cases. $u^{ij}_{~kl}$ is either $u^{i1}_{~j1}$ or $u^{ij}_{~11}$. Notice that $u^{ij}_{~11} = 0$. And $u^{i1}_{~j1} = 0$ unless $i=j$.

\item There are exactly three of $i,j,k,l$ which equals to 1. Then we can assume that $u^{ij}_{~kl}$ is $u^{i1}_{~11} = 0$.

\item $i=j=k=l=1$. This is obvious.
\end{enumerate}
\end{proof}

As a corollary, we can simplify the expression of $|Rm|$ at $(x_1, 0, \ldots, 0)$.

\begin{cor}
$$
|Rm|^2(x_1, 0, \ldots, 0) = \sum_{i} (u^{ii}_{~ii})^2 + 4 \sum_{i < j} (u^{ij}_{~ij})^2 .
$$
\end{cor}

Together with the above lemmas, we have

\begin{prop}
\label{grad}
There exists a constant $C > 0$ depending only on $C_1, P$ such that
$$
\left|\frac{\partial f}{\partial x_1} (x_1, 0,\ldots,0) - \frac{\partial f}{\partial x_1} (0, 0,\ldots,0) \right|< C, ~ x_1 \in [0,1].
$$
\end{prop}

\begin{proof}
Let $V(x_1) = \frac{1}{2} x_1 \ln x_1 + f(x_1, 0, \ldots, 0)$. It is easy to see that 
$$
\left( \frac{1}{V''} \right)' (0) = 2.
$$
So
\begin{eqnarray*}
\left| \left( \frac{1}{V''} \right)' (s) - 2 \right| = \left| \int_0^s \left( \frac{1}{V''} \right)'' (x) \ dx \right| \leq s.
\end{eqnarray*}
Hence
$$
2 - s \leq  \left( \frac{1}{V''} \right)' (s) \leq  2+ s.
$$
Since 
$$
\frac{1}{V''}(0) = 0,
$$
we conclude that
$$
2s - \frac{1}{2}s^2 \leq \frac{1}{V''}(s) \leq 2s + \frac{1}{2}s^2.
$$

In terms of $f$, we have
$$
-\frac{1}{8+2 s} \leq \frac{\partial^2 f}{\partial x_1^2} (s, 0)  \leq \frac{1}{8 -2 s}.
$$

Thus we obtain the conclusion.
\end{proof}

Our next few lemmas show that we can also control $| \frac{\partial f}{\partial x_2}(x_1, 0) - \frac{\partial f}{\partial x_2}(0, 0)|$.

\begin{lem} For any $x_1 \in (0,1]$, we have
$$
u^{12}_{~12}(x_1, 0) = - \frac{\partial \frac{2 f_{12}}{f_{11} + \frac{1}{2 x_1}}}{\partial x_1} (x_1, 0)
$$
\end{lem}

\begin{proof}
This is done by direct calculations. Let $A$ be the matrix of $(u_{ij})$ without the first column and the second row. Then

\begin{eqnarray*}
u^{12}_{~12}(x_1, 0) & = & - \lim_{x \rightarrow (x_1, 0,\ldots, 0)} \frac{\partial^2 \frac{A}{\det(u_{ij})}}{\partial x_1 \partial x_2}(x)\\
& = & - \frac{\partial \frac{2 f_{12}}{f_{11} + \frac{1}{2 x_1}}}{\partial x_1} (x_1, 0)\\
\end{eqnarray*}
\end{proof}

\begin{prop}
\label{grad2}
For any $x_1 \in (0,1]$, we have
$$
\left|\frac{\partial f}{\partial x_2} (x_1, 0)  - \frac{\partial f}{\partial x_2} (0, 0)\right| < C,
$$
where $C$ depends only on $C_1$ and $P$.
\end{prop}

\begin{proof}
Combining the previous results, we have
$$
 \left| \frac{\partial \frac{2 f_{12}}{f_{11} + \frac{1}{2 x_1}}}{\partial x_1} (x_1, 0) \right| < 1.
$$

Notice that  
$$
\lim_{x_1 \rightarrow 0}  \frac{2 f_{12}}{f_{11} + \frac{1}{2 x_1}} (x_1, 0) = 0.
$$

Then

\begin{eqnarray*}
\left|  \frac{2 f_{12}}{f_{11} + \frac{1}{2 x_1}} (x_1, 0) \right| < x_1.
\end{eqnarray*}

Thus we have
$$
- x_1 \left(f_{11} + \frac{1}{2x_1}\right) <2 f_{12}(x_1,0) <  x_1 \left(f_{11} + \frac{1}{2x_1}\right).
$$

To obtain our conclusion, we need to calculate 
\begin{eqnarray*}
& &\left| \int_0^{x_1} s f_{11}(s, 0) ~ ds \right| \\
&=&\left| x_1 f_1(x_1,0) - \int_0^{x_1} f_1(s,0) ~ ds \right| \\
&=&\left| x_1 (f_1(x_1,0) - f_1(0,0)) - \int_0^{x_1} (f_1(s,0) - f_1(0,0)) ~ ds \right| \\
&<& C x_1.
\end{eqnarray*}
\end{proof}

Let us assume that $u$ reaches its maximum at $O = (0,\ldots,0)$. Let 
$$
\left|\frac{\partial f}{\partial x_i} (O) \right| = B_i.
$$ We further assume that $$B_1 \geq \cdots \geq B_n.$$ We now try to control $B_1$.

\begin{prop}
\label{gb}
If $\frac{\partial f}{\partial x_1} (O) > 0$, then $B_1$ is bounded by a constant depending only on $C_1$ and $P$.
\end{prop}

\begin{proof}
Let us pick a point $x_0 =\left(\frac{n-1}{n}, \frac{1}{(n-2)n}, \ldots, \frac{1}{(n-2)n},0 \right)$. We parameterize the segment $l$ connecting $O$ and $x_0$ by 
\begin{eqnarray*}
l(t) =  \left(\frac{n-1}{n}t, \frac{1}{(n-2)n}t, \ldots, \frac{1}{(n-2)n}t,0 \right), \quad 0 \leq t \leq 1.
\end{eqnarray*}

Let 
\begin{eqnarray*}
V(t) &=& u(l(t)) \\
&=& \frac{1}{2} \left(\frac{n-1}{n}t \ln \left(\frac{n-1}{n}t \right) + (n-2) \frac{1}{(n-2)n}t \ln \left(\frac{1}{(n-2)n}t \right) \right)\\
& & +f(l(t)) \\
&=&  \frac{1}{2} t \ln t + g(t),
\end{eqnarray*} 
where $g(t) = f(l(t)) +  \frac{1}{2} \left(\frac{n-1}{n}t \ln \left(\frac{n-1}{n} \right) - \frac{1}{n}t \ln (n-2)n \right) $.
Then 
$$
g'(0) \geq \frac{n-1}{n}B_1 - (n-2) \frac{1}{(n-2)n}B_1 + C_0 =  \frac{n-2}{n}B_1 + C_0,
$$ where $C_0 =\frac{1}{2} \left(\frac{n-1}{n} \ln \left(\frac{n-1}{n} \right) - \frac{1}{n} \ln (n-2)n \right) $.

By Lemma 3 of \cite{D3} and Lemma 4.3 of \cite{H1}, we have
$$
\left( \frac{1}{V''} \right)''(t) \leq |Rm|(l(t)) \leq 1.
$$

Following the calculations in Proposition (\ref{grad}), we have
$$
g''(t) \geq -\frac{1}{8+2 t}, \quad t \in [0,1]. 
$$

Thus
\begin{eqnarray*}
g'(1) &=& g'(0) + \int_0^1 g''(t) ~ dt\\
&\geq&  \frac{n-2}{n}B_1 + C_0 - \int_0^1 \frac{1}{8+2 t} ~ dt  \\
&\geq&  \frac{n-2}{n}B_1 + C.
\end{eqnarray*}

Then
\begin{eqnarray*}
g(1) &=& g(0) + \int_0^1 g'(t) ~ dt\\
&\geq& \frac{n-2}{n}B_1 + C + g(0).
\end{eqnarray*}

Since $g(0) = V(0) \geq V(1) = g(1)$, we conclude that $B_1$ is bounded.
\end{proof}

The second case we should deal with is $\frac{\partial f}{\partial x_1} (0) < 0$. Before we proceed, we observe the following lemma.

\begin{lem}
\label{bound}
Let $Q$ be the facet of $P$ in the $x_1 \cdots x_{n-1}$ plane. For every $\epsilon > 0$, let $Q_\epsilon$ be the set of points in $Q$ whose distance is at least $\epsilon$ away from $\partial Q$. Then there exists a constant $C$ depending only on $\epsilon, C_1$ and $P$ such that for any $x \in Q$
$$
u(x) < C, \quad \left| \frac{\partial u}{\partial x_1} \right| (x) < C, \ldots, \left| \frac{\partial u}{\partial x_{n-1}} \right| (x) < C.
$$
\end{lem}

\begin{proof}
Let $v$ be the restriction of $u$ on $Q$.  For any $\epsilon$ and any $x=(x_1, \ldots, x_{n-1}) \in Q_\epsilon$, we want to show that $v(x)$ is bounded by a constant depending only on $\epsilon, C_1$ and $P$. If $\nabla v(x) = (0,\ldots,0)$, then for any $y \in Q$, $v(y) \geq v(x) > 0$ by convexity. Thus $v(x)$ is controlled by $C_1$ and $P$.

Suppose $\nabla v(x) \neq (0,\ldots,0)$. Let $l$ be the hyperplane whose normal vector is $\nabla v(x)$ and $l(x) = 0$. Then for any point $y \in \{l \geq 0\} \cap Q$, we have $v(y) \geq v(x)$ by convexity. Thus $v(x)$ is controlled by $\epsilon, C_1$ and $P$. It is easy to see that $\nabla v(x)$ is also controlled.
\end{proof}

\begin{prop}
\label{gb2}
Suppose $\frac{\partial f}{\partial x_1} (O) < 0$. Then $B_1$ is also bounded by a constant depending only on $C_1$ and $P$.
\end{prop}

\begin{proof}
Let $y_1 = (1, 0, \ldots, 0),~ y_2 = (0, \frac{1}{2(n-2)}, \ldots, \frac{1}{2(n-2)}, 0)$. We parameterize the line connecting $y_1$ and $y_2$ by $l(t) = (1 - 2t,  \frac{1}{n-2} t,\ldots,\frac{1}{n-2} t, 0),~ t \in [0, \frac{1}{2}]$. Let
\begin{eqnarray*}
V(t) &=& u(l(t)) \\
&=& \frac{1}{2} \left( (1-2t) \ln(1-2t) + t \ln \frac{1}{n-2} t \right) + f(l(t)) \\
&=& \frac{1}{2} t \ln t + g(t),
\end{eqnarray*}
where $g(t) = f(l(t)) + \frac{1}{2} ((1-2t) \ln(1-2t) - t \ln (n-2))$. 

Notice that by Proposition (\ref{grad}), we have $\frac{\partial f}{\partial x_1} (y_1) \leq -B_1 + C$. By Proposition (\ref{grad2}), we have $|\frac{\partial f}{\partial x_i} (y_1) | \leq B_i + C \leq B_1 + C,~ 2 \leq i \leq n.$

Then
$$
g'(0) \geq 2 B_1 - B_1 + C \geq B_1 + C.
$$
Again we have
$$
g''(t) \geq - \frac{1}{8+2t}, \quad t \in [0,\frac{1}{2}].
$$
Thus
$$
g'(t) \geq B_1 + C, \quad t \in [0,\frac{1}{2}].
$$
So
$$
g(\frac{1}{4}) \geq g(0) + \frac{1}{4} (B_1 + C).
$$
Notice that $g(0) = V(0) \geq 0$. And $g(\frac{1}{4}) = V(\frac{1}{4}) + \frac{1}{8} ln 4 $ is bounded by Lemma (\ref{bound}). We obtain that $B_1$ is bounded.
\end{proof}

We are ready to provide a proof of Theorem (\ref{dia}).

\begin{proof}[Proof of Theorem (\ref{dia}).]
We parameterize the line connecting $O$ to $(\frac{1}{n-1},\ldots,\frac{1}{n-1},0)$ by
$$
l(t) = \left(\frac{1}{n-1}t, \ldots, \frac{1}{n-1}t, 0 \right), \quad t \in [0,1].
$$
Let 
\begin{eqnarray*}
V(t) &=& u(l(t)) \\
&=& \frac{1}{2} \left( t \ln \frac{1}{n-1}t \right)  + f(l(t)) \\
&=& \frac{1}{2}  t \ln t + g(t),
\end{eqnarray*}
where $g(t)  = f(l(t)) - \frac{1}{2}t \ln (n-1)$. Notice that $g'(0)$ is bounded by Proposition (\ref{gb}) and (\ref{gb2}). By Lemma (\ref{bound}), $g'(1)$ is bounded. Similar calculations in Proposition (\ref{grad}) show that $$g''(t) \geq - \frac{1}{8+2t}.$$ Thus we obtain  that $g'(t)$ is bounded for $t \in [0,1]$. Since $V(1)$ is bounded by Lemma (\ref{bound}), we have that $V(0)$ is bounded.
\end{proof}

\section{Proof of Theorem (\ref{main})}

Let $u^{(\alpha)}(t, x), ~ t \in [-1,0], x \in P$ be a sequence of modified Calabi flows  on $P$ satisfying:

\begin{itemize}

\item For any $\alpha$, the Riemannian curvature of $u^{(\alpha)}(t,x), ~ t \in [-1, 0], x \in P$ is bounded by $ C_1$.  

\item For any $\alpha$, $u^{(\alpha)}(0,x)$ is normalized and
$$
\int_{\partial P} u^{(\alpha)}(0,x) ~ d \sigma < C_2.
$$
\end{itemize}

By the results of the previous section, we conclude that $u^{(\alpha)}(0,x)$ is bounded for all $\alpha$. Thus $u^{(\alpha)}(t,x)$ is bounded for all $\alpha$ and $t \in [-1,0]$. Let us fix a toric invariant background metric $\omega$ and let its symplectic potential be $u$. Also we let $\varphi^{(\alpha)}(t)$ be the difference of the Legendre transform of  $u^{(\alpha)}(t)$ and $u$. Then Donaldson \cite{D4} shows that
$$
\| \varphi^{(\alpha)}(t) \|_{L^\infty} = \| u^{(\alpha)}(t) - u \|_{L^\infty}.
$$
Since the curvature is bounded, applying Theorem 5.1 of \cite{ChenHe}, we obtain that the metric $\omega_{\varphi^{(\alpha)}(t)}$ is equivalent to $\omega$ and the $C^{3,\alpha}$ norm of $\varphi^{(\alpha)}(t), t \in [-1, 0]$ is uniformly bounded. Thus we can smooth $\varphi^{(\alpha)}(t)$ for $t \in [-\frac{1}{2},0]$ by Theorem 3.3 of \cite{ChenHe}.

Our discussions lead to the following result:

\begin{prop} 

\label{com}

By passing to a subsequence, $\varphi^{(\alpha)}(t)$ converges a limiting modified Calabi flow $\varphi(t),~ t\in [-\frac{1}{2}, 0]$.
\end{prop}

Now we are ready for the proof of our main theorem. 

\begin{proof}[Proof of Theorem (\ref{main}).]

Let $u(t,x)$ be a one parameter group of symplectic potentials satisfying the Calabi flow equation and the Riemannian curvature is uniformly bounded along the flow. Then the corresponding modified Calabi flow is $u(t,x) + t (\theta - \underline{R})$. It is easy to see that the Riemannian curvature is uniformly bounded along the modified Calabi flow. For convenience, we still denote $u(t,x)$ as the modified Calabi flow. Notice that the modified Calabi flow is the downward gradient flow of the modified Mabuchi energy. Since $(X,P)$ is analytic uniform $K$-stable, Proposition 5.1.2 of \cite{D1} tells us that the modified Mabuchi energy is bounded from below. Moreover, Proposition 5.1.8 of \cite{D1} tells us that there exists $C_1 > 0$ such that for each $t$
$$
\int_{\partial P} \tilde{u} (t, x) ~ d\sigma < C_1,
$$
where $\tilde{u}(t,x)$ is the normalization of $u(t,x)$.

Let us pick a sequence of time $t_i \rightarrow \infty$ and we define a sequence of modified Calabi flows as 
$$
u^{(\alpha)}(t, x) = u(t_{\alpha} + t, x),~ t \in [-1,0].
$$

We add some affine function $l^{(\alpha)}(x)$ to $u^{(\alpha)}(t,x), ~ t \in [-1,0]$ so that $u^{(\alpha)}(0,x)$ is normalized. We still denote the new modified Calabi flow as $u^{(\alpha)}(t,x), ~ t \in [-1,0]$. Let us pick a background symplectic potential $u(x)$. Let $\varphi^{(\alpha)}(t)$ be the K\"ahler potential which is the difference of the Legendre transform of $u^{(\alpha)}(t)$ and $u$. Proposition (\ref{com}) shows that after passing to a subsequence, the modified Calabi flows $\varphi^{(\alpha)}(t), ~ t \in [-1,0]$ converges to a limiting modified Calabi flow $\varphi^{(\infty)}(t), ~ t \in [-\frac{1}{2}, 0]$. 

For every $t \in [-\frac{1}{2}, 0]$,  the modified Mabuchi energy of $\varphi^{(\infty)}(t)$ is the infimum of the modified Mabuchi energy of $u(t,x)$. So we conclude that for every $t \in [-\frac{1}{2}, 0]$,  $\varphi^{(\infty)}(t)$ is an extremal metric. Let $u^{(\infty)}(t)$ be the corresponding symplectic potential of $\varphi^{(\infty)}(t),~ t \in [-\frac{1}{2}, 0]$.

Since $\varphi^{(\alpha)}(0)$ converges to $\varphi^{(\infty)}(0)$, the geodesic distance between $\varphi^{(\alpha)}(0)$ and $\varphi^{(\infty)}(0)$ goes to zero as $\alpha$ goes to infinity, i.e.,
$$
\lim_{\alpha \rightarrow \infty} \int_P | u^{(\alpha)}(0) - u^{(\infty)}(0) |^2 ~ d \mu = 0.
$$
It shows that there exists a constant $C > 0$ such that
$$
 \int_P | u^{(\alpha)}(0)|^2 ~ d \mu < C.
$$
Calabi and Chen \cite{CC} show that the geodesic distance is decreasing under the Calabi flow. It is easy to see that  the geodesic distance is decreasing under the modified Calabi flow, i.e.,
$$
\frac{d}{d t} \int_P | u(t) - u^{(\infty)}(0)|^2 ~ d \mu \leq 0.
$$
We conclude that
$$
 \int_P | u(t)|^2 ~ d \mu < C,
$$
for all $t$. Thus $l^{(\alpha)}(x)$ is a bounded affine function on $P$. An immediate consequence is that the $C^0$ norm of $u(t_i)$ is uniformly bounded. Let $\varphi(t)$ be the difference of the Legendre transform of of $u(t)$ and $u$. We conclude that the $C^0$ norm of $\varphi(t_i)$ is uniformly bounded. We construct a sequence of modified Calabi flows as
$$
\phi^{(\alpha)}(t) = \varphi(t_i+t), ~ t \in [-1, 0].
$$
Then the $C^0$ norm of $\phi^{(\alpha)}(t)$ is uniformly bounded for all $\alpha$ and $t \in [-1, 0]$. Since the curvature of $\phi^{(\alpha)}(t)$ is also uniformly bounded for all $\alpha$ and $t \in [-1, 0]$, applying Theorem 5.1 of \cite{ChenHe}, we conclude that the metric $\omega_{\phi^{(\alpha)}(t)}$ is equivalent to $\omega$ and the $C^{3,\alpha}$ norm of $\phi^{(\alpha)}(t), t \in [-1, 0]$ is uniformly bounded. Thus we can smooth $\phi^{(\alpha)}(t)$ for $t \in [-\frac{1}{2},0]$ by  Theorem 3.3 of \cite{ChenHe}. Hence by passing to a subsequence, we obtain a limiting modified Calabi flow $\phi^{(\infty)}(t)$ for $t \in [-\frac{1}{2},0]$. Again, the modified Mabuchi energy of $\phi^{(\infty)}(t),~ t \in [-\frac{1}{2},0]$ are the same because each of them is the infimum of the modified Mabuchi energy of the modified Calabi flow $\varphi(t)$. So we conclude that $\phi^{(\infty)}(t),~ t \in [-\frac{1}{2},0]$ are extremal metrics. Thus $\varphi(t_i)$ converges to the extremal metric $\phi^{(\infty)}(0)$. By applying the results of \cite{HZ}, we conclude that the modified Calabi flow $\varphi(t)$ converges to the extremal metric $\phi^{(\infty)}(0)$ exponentially fast.
\end{proof}

\section{Proof of Theorem (\ref{proof of cos})}

Let $\mathbb{G}$ be a maximal compact subgroup of the identity component of the reduced automorphism group of $(X,J)$, i.e., Aut$_0 (X,J)$. Also we let  $\mathbb{T}$ be a maximal torus of $\mathbb{G}$. Suppose $\omega_0$ is an extremal metric invariant under $\mathbb{G}$. Since $(X,J) = P(E) \rightarrow \Sigma$ and $\Sigma$ is a curve of genus $\geq 2$, we conclude that Aut$_0(X,J) \cong H^0(\Sigma, PGL(E)).$ By Lemma 1 of \cite{ACGT2}, $E$ decomposes as a direct sum
$
E = \bigoplus_{i=0}^l E_i,
$
where $E_i$ is indecomposable and $l=\dim(\mathbb{T})$. 

Next we introduce the notion of admissible K\"ahler metrics. An admissible K\"ahler metric $g$ is a type of the generalized Calabi construction \cite{ACGT2}: 

\begin{enumerate}

\item $g$ is invariant under $\mathbb{T}$.

\item Let $\omega$ be the symplectic form of $g$ and $z$ be the moment map:
$$
z : X \rightarrow \mathfrak{t}^*,
$$
where $\mathfrak{t}$ is the Lie algebra of $\mathbb{T}$. $g$ is also rigid with respect to $\mathbb{T}$, i.e., for any $x \in X$, $\mathfrak{i}_x^* g$ depends only on $z(x)$, where
$
\mathfrak{i}: \mathbb{T} \rightarrow \mathbb{T} \cdot x \subset X
$
is the orbit map.

\item The image of $X$ under the moment map is a Delzant polytope $P$. Notice that in our case, we can associate $P$ with a toric variety $(V, g_V, J_V) = (\mathbb{C}P^l, g_V, J_V)$. Since the torus action is rigid, we obtain the smooth complex quotient $\hat{S} \cong X /\mathbb{T}^c$. Moreover, $X^0 = z^{-1}(P^0)$ is a principal $\mathbb{T}^c$ bundle over $\hat{S}$, where $P^0$ is the interior of $P$. In fact, $(X, J)$ can be obtained by blowing down 
$
\hat{X} = X^0 \times_{\mathbb{T}^c} V \rightarrow \hat{S}
$
along the inverse images of facets of $P$.  In our case, we can write $\hat{S} = P(E_0) \times_\Sigma \cdots \times_\Sigma P(E_l) \rightarrow \Sigma$, where each $E_i \rightarrow \Sigma$ is a projectively-flat hermitian bundle with rank $d_i+1$. Thus $\hat{X} = P(\mathcal{O}(-1)_{E_0} \oplus \cdots \oplus \mathcal{O}(-1)_{E_l}) \rightarrow \hat{S}$, where $\mathcal{O}(-1)_{E_i}$ is the (fibrewise) tautological line bundle over $P(E_i) \rightarrow \Sigma$. 

\item $g$ is also semisimple (see Definition 2 of \cite{ACGT2}). Then we conclude that the universal cover of $(\hat{S}, J_{\hat{S}}, g_{\hat{S}})$ is the product of K\"ahler manifolds 
$$
\prod_{i=0}^l (\mathbb{C}P^{d_i}, J_i, g_i) \times (\Sigma, J_\Sigma, g_\Sigma),
$$
where $g_i$ is a Fubini-Study metric with scalar curvature $2d_i(d_i+1)$ and $g_\Sigma$ is a cscK metric. Let $\omega_i$ be the K\"ahler form of $g_i$ and $\omega_\Sigma$ be the K\"ahler form of $g_\Sigma$. We fix a
hermitian metric on $E_i$ whose Chern connection has curvature 
$\Omega_i \otimes Id_{E_i}$ with
$$
\Omega_i - \Omega_0  = p_i \omega_\Sigma, \quad i \geq 1.
$$
We also let $p_\Sigma = (p_1, \ldots , p_l) \in \mathbb{R}^l \cong \mathfrak{t}$. Now we want construct $\hat{\theta} = (\hat{\theta}_1, \ldots, \hat{\theta}_l)$ to be a principle $\mathbb{T}$-connection associated with the principle $\mathbb{T}^c$-bundle $X^0$ over $\hat{S}$. We let $\hat{\theta}_i$ be a connection 1-form for the principal $U(1)$-bundle over $\hat{S}$, associated to the line bundle $\mathcal{O}(−1)_{E_i}$, with curvature $d \hat{\theta}_i = -\omega_i + \Omega_i$. 

\item Following the notation of \cite{ACGT2}, we write the toric invariant metric on $V=\mathbb{C}P^l$ as
$$
g_V = \langle dz, G, dz \rangle + \langle dt, H, dt \rangle,
$$
where $G$ is a positive definite $S^2 \mathfrak{t}$-valued function on $P^0$, $H$ is its inverse in $S^2 \mathfrak{t}^*$. Moreover, $\langle \cdot, \cdot, \cdot \rangle$ denotes the point-wise contraction $\mathfrak{t}^* \times S^2 \mathfrak{t} \times \mathfrak{t}^* \rightarrow \mathbb{R}$ or the dual contraction. On $P^0$, the K\"ahler metric $g$ can be written as 
\begin{equation}
\label{cm}
g = ( \langle p_\Sigma, z \rangle + c_\Sigma) g_\Sigma +  \sum_{j=0}^l ( \langle p_j, z \rangle + c_j) g_j + \langle dz, G, dz \rangle + \langle \hat{\theta}, H, \hat{\theta} \rangle,
\end{equation}
where $p_j, 0 \leq j \leq l$ is the inward normal vector of the facets $F_j$ of $P$ and we choose $c_j$ such that $\langle p_j, z \rangle + c_j = 0$ on $F_j$.  We also choose $c_\Sigma$ such that $\langle p_\Sigma, z \rangle + c_\Sigma$ is positive on $P$. 
\end{enumerate}

A K\"ahler metric $g$ in the form of (\ref{cm}) is an admissible metric. A K\"ahler class $\Omega$ is admissible if $\Omega$ contains an admissible metric. Notice that $G=(u_{ij}), H=(u^{ij})$, where $u$ is a symplectic potential on $P$ satisfying the Guillemin boundary conditions. The expression of the scalar curvature is calculated in \cite{ACGT2}:
$$
R_g = \frac{Scal_\Sigma}{\langle p_\Sigma, z \rangle + c_\Sigma} +  \sum_{j=0}^l \frac{Scal_j}{\langle p_j, z \rangle + c_j} - \frac{1}{p(z)} \sum_{r,s=1}^l \frac{\partial^2}{\partial z_r \partial z_s} (p(z)u^{rs}),
$$
where $p(z)=(\langle p_\Sigma, z \rangle + c_\Sigma) \prod_{j=0}^l (\langle p_j, z \rangle + c_j)^{d_j}.$ And the volume form is
$$
\omega^n = p(z) \omega_\Sigma \wedge \left( \bigwedge_{j=0}^l \omega_j^{\wedge d_j} \right) \wedge \langle dz \wedge \hat{\theta} \rangle^l,
$$
where $n$ is the dimension of $X$.

Since $Scal_\Sigma$ and $Scal_j, 0 \leq j \leq l$ are constants, we can run the Calabi flow on the symplectic potential level, i.e.,
$$
\frac{\partial u}{\partial t} = \underline{R} - R_u.
$$
The modified Calabi flow on the symplectic potential level is
$$
\frac{\partial u}{\partial t} = - R_u^{\perp},
$$
where 
$$
R_u^{\perp} = \langle A, z \rangle + B + R_u.
$$
We refer to \cite{ACGT2} for the detailed expressions of the a-priori determined vector (or constant) $A$ (or $B)$. We start the Calabi flow from an admissible K\"ahler metric $\omega$. Let $u(t)$ be a one parameter group of symplectic potentials corresponding to a modified Calabi flow. Also let $\varphi(t)$ be the K\"ahler potential which is the difference of the Legendre transform of $u(t)$ and $u(0)$. We express the extremal metric $\omega_0$ as $\omega_0 = \omega + \sqrt{-1} \partial \bar{\partial} \varphi_0$.  Notice that the modified Calabi flow is just the Calabi flow pulling back by the extremal vector field \cite{HZ}. Since the Calabi flow decreases the distance, we conclude that the modified Calabi flow also decreases the distance, i.e., $d(t) = d(\varphi_0, \varphi(t))$ is a decreasing function in terms of $t$. Thus
$$
d(u(0), u(t)) = d(0, \varphi(t)) \leq d(0, \varphi_0) + d(\varphi_0, \varphi(t)) < C.
$$
Let $v_t(s) = (1-s) u(0) + s u(t), 0 \leq s \leq 1$ and $\phi_t(s)$ be the K\"aler potential which is the difference of Legendre transform of $v_t(s)$ and $u(0)$. Then $\phi_t(s)$ is the geodesic connecting $\varphi(0)$ and $\varphi(t)$. Moreover,
$$
d(u(0), u(t)) = \sqrt {\int_X (u(t)-u(0))^2 ~ \omega^n } = \sqrt {\int_P (u(t)-u(0))^2 ~ p(z) d \mu } 
$$
Thus we conclude that there exists a constant $C > 0$ such that for any $t > 0$, 
$$
\int_P u(t)^2 ~ p(z) d \mu < C.
$$
Then we have the following proposition.

\begin{prop}
\label{ib}
For every $p \in P^0$, there exists a constant $C(p)$ depending only on the Euclidean distance between $p$ and $\partial P$ such that for any $t > 0$,
$$
|u|(t, p) < C(p), \quad |D u|(t,p) < C(p),
$$
where $D  u$ is the Euclidean derivative of $u$.
\end{prop}

\begin{proof}
Let $P_\epsilon$ be the sets of points in $P$ whose Euclidean distance to $\partial P$ is at lease $\epsilon$. Since $p(z)$ is positive in the interior of $P$. We obtain
$$
\int_{P_\epsilon} u(t)^2 d \mu < C(\epsilon),
$$
where $C(\epsilon)$ is a constant depending only on $\epsilon$. Applying Proposition 3.4 of \cite{FH}, we know that $u(t)$ is bounded in the interior of $P_\epsilon$. Thus the Euclidean derivative $D u(t)$ is also bounded in the interior of $P_\epsilon$ by the convexity of $u(t)$.
\end{proof} 

\begin{proof}[Proof of Theorem (\ref{proof of cos})]
Let us assume that our Calabi flow starting from an admissible metric $\omega$. Let $u(t)$ be the corresponding symplectic potential of the modified Calabi flow. Proposition 8 of \cite{ACGT1} shows that each leave $V=\mathbb{C}P^l$ is totally geodesic in $X$. Thus the curvature is bounded on $X$ implies that the curvature is bounded on $P$, i.e., there exists a constant $C > 0$ such that for any $t > 0,~ p \in P$,
$$
\sum_{i,j,k,l} u(t)^{ij}_{~kl} u(t)^{kl}_{~ ij} (p) < C.
$$
Proposition (\ref{ib}) and the proof of Theorem (\ref{dia}) tell us that there exists a constant $C > 0$ such that for any $t > 0$, 
$$
|u(t)|_{L^\infty} < C.
$$ 
Then Donaldson \cite{D4} shows that
$$
|\varphi(t)|_{L^\infty} < C.
$$
Applying Theorem 5.1 of \cite{ChenHe}, we conclude that the metric $\omega_{\varphi(t)} = \omega+ \sqrt{-1} \partial \bar{\partial} \varphi(t)$ is equivalent to $\omega$ and the $C^{3,\alpha}$ norm of $\varphi(t)$ is uniformly bounded. By the smoothing property of the Calabi flow, i.e., Theorem 3.3 of \cite{ChenHe}, we can uniformly control the $C^\infty$ norm of $\varphi(t)$. Corollary 2.2 of \cite{HZ} shows that the modified Calabi flow converges to an extremal metric. Moreover, Theorem 1.5 of \cite{HZ} shows that the convergence is exponentially fast.
\end{proof}

Hongnian Huang, \ hnhuang@gmail.com

CMLS

Ecole Polytechnique

\begin{thebibliography}{10}

\bibitem{ACGT1} V. Apostolov,~ D.M.J. Calderbank, ~P. Gauduchon and C.W. T{\o}nnesen-Friedman, {\em Hamiltonian 2-forms in K\"ahler Geometry I: General Theory}, J. Differential Geom. 73 (2006), 359-412.

\bibitem{ACGT} V. Apostolov,~ D.M.J. Calderbank, ~P. Gauduchon and C.W. T{\o}nnesen-Friedman, {\em Hamiltonian 2-forms in K\"ahler geometry III: Extremal metrics and stability}, Invent. Math. 173 (2008), 547-601.

\bibitem{ACGT2}  V. Apostolov,~ D.M.J. Calderbank, ~P. Gauduchon and C.W. T{\o}nnesen-Friedman, {\em Extremal K\"ahler metrics on projective bundles over a curve}, Adv. Math. 227 (2011), 2385-2424.

\bibitem{Ca1} E. Calabi, {\em Extremal K\"ahler metric, in Seminar of Differential Geometry}, ed. S. T. Yau, Annals of Mathematics Studies 102, Princeton University Press (1982), 259-290.

\bibitem{Ca2} E. Calabi, {\em Extremal K\"ahler metric, II, in Differential Geometry and Complex Analysis}, eds. I. Chavel and H. M. Farkas, Spring Verlag (1985), 95-114.

\bibitem{CC} E. Calabi and X.X. Chen, {\em Space of K\"ahler metrics and Calabi flow}, J. Differential Geom. 61 (2002), no. 2, 173-193.

\bibitem{CLS} B.H. Chen, A.M. Li and L. Sheng, {\em Uniform K-stability for extremal metrics on toric varieties}, arXiv:1109.5228.

\bibitem{Ch1} X.X. Chen, {\em Calabi flow in Riemann surfaces revisited: a new point of view}, Internat. Math. Res. Notices 2001, no. 6, 275-297.

\bibitem{ChenHe} X.X. Chen and W.Y. He, {\em On the Calabi flow}, Amer. J. Math. 130 (2008), no. 2, 539-570.

\bibitem{ChenHe2} X.X. Chen, W.Y. He, {\em The Calabi flow on K\"ahler surface with bounded Sobolev constant--(I)}, arXiv:0710.5159.

\bibitem{ChenHe3} X. X. Chen and W. Y. He, {\em The Calabi flow on toric Fano surface}, arXiv:0807.3984.

\bibitem{CZ} X.X. Chen and M.J. Zhu, {\em Liouville energy on a topological two sphere},  arXiv:0710.4320.

\bibitem{D1} S.K. Donaldson, {\em Scalar curvature and stability of toric varieties}, Jour. Differential Geometry 62 (2002), 289-349.

\bibitem{D2} S.K. Donaldson, {\em Interior estimates for solutions of Abreu's equation}, Collectanea Math. 56 (2005), 103-142.

\bibitem{D3} S.K. Donaldson, {\em Extremal metrics on toric surfaces: a continuity method}, J. Differential Geom. 79 (2008), no. 3, 389-432.

\bibitem{D4} S.K. Donaldson, {\em Constant scalar curvature metrics on toric surfaces}, Geom. Funct. Anal. 19 (2009), no. 1, 83-136.

\bibitem{D5} S.K. Donaldson, {\em b-Stability and blow-ups},  arXiv:1107.1699.

\bibitem{D6} S.K. Donaldson, {\em Conjectures in K\"ahler geometry}, Strings and geometry, 71-78, Clay Math. Proc., 3, Amer. Math. Soc., Providence, RI, 2004.

\bibitem{D7} S.K. Donaldson, {\em Lower bounds on the Calabi functional}, J. Differential Geom., 70(3):453-472, 2005.

\bibitem{D8} S.K. Donaldson, {\em K\"ahler geometry on toric manifolds, and some other manifolds with large symmetry}, Handbook of geometric analysis. No. 1, 29-75, Adv. Lect. Math. (ALM), 7, Int. Press, Somerville, MA, 2008.

\bibitem{FH} R.J. Feng and H.N. Huang, {\em The Global Existence and Convergence of the Calabi Flow on $\mathbb{C}^n = \mathbb{Z}^n + i \mathbb{Z}^n$}, J. Funct. Anal. (2012), http://dx.doi.org/10.1016/j.jfa.2012.05.017

\bibitem{H1} H.N. Huang, {\em On the Extension of the Calabi Flow on Toric Varieties},  Ann. Global Anal. Geom. 40 (2011), no. 1, 1-19, arxiv:1101.0638.

\bibitem{H2} H.N. Huang, {\em Toric Surface, $K$-Stability and Calabi Flow}, Preprint.

\bibitem{HZ} H.N. Huang and K. Zheng, {\em Stability of Calabi flow near an extremal metric},  Ann. Sc. Norm. Super. Pisa Cl. Sci. (5) 11 (2012), no. 1, 167-175, arXiv:1007.4571.

\bibitem{Ra} A.A. Raza, {\em Scalar curvature and multiplicity-free actions, PhD thesis}, Imperial College London, 2005.

\bibitem{St} J. Streets, {\em The long time behavior of fourth-order curvature flows}, to appear in Calc. Var. PDE.

\bibitem{S1} G. Sz\'ekelyhidi, {\em Optimal test-configurations for toric varieties}, J. Differential Geom. 80 (2008), 501-523.

\bibitem{S2} G. Sz\'ekelyhidi, {\em Filtrations and test-configurations}, arXiv:1111.4986.

\bibitem{S3} G. Sz\'ekelyhidi, {\em Extremal metrics and $K$-stability}, Ph.D thesis, Imperial college.

\bibitem{Ti1} G. Tian, {\em K\"ahler-Einstein metrics of positive scalar curvature}, Inventiones Math. 130 1-57 (1997)

\bibitem{T1} V. Tosatti, {\em K\"ahler-Ricci flow on stable Fano manifolds}, J. Reine Angew. Math. 640 (2010), 67-84. 

\bibitem{Y1} S.T. Yau, {\em Review of K\"ahler-Einstein metrics in algebraic geometry}, Israel
Math. Conference Proc., Bar-Ilan Univ. 9 433-443 (1996)

\end{thebibliography}
\end{document}